\documentclass[a4paper,11pt]{amsart}
\usepackage[latin2]{inputenc}
\usepackage{latexsym}
\usepackage{amsmath,mathtools, amssymb, amsthm}
\usepackage{enumerate}
\usepackage[english]{babel}
\usepackage{color}
\usepackage{subfigure}
\usepackage{float}
\usepackage[normalem]{ulem}
\usepackage{braket}

\newtheorem{theorem}{Theorem}[section]
\newtheorem{corollary}{Corollary}

\newtheorem{lemma}[theorem]{Lemma}
\newtheorem{proposition}{Proposition}

\theoremstyle{definition}

\newtheorem{remark}{Remark}

\newtheorem{bsp}[theorem]{{\bf Example}}
\newenvironment{example}{\begin{bsp}\rm}{\end{bsp}}

\def\mv{\mathsf{v}}
\def\me{\mathsf{e}}

\def\bB{\mathbb{B}}
\def\bM{\mathbb{M}}
\def\RR{\mathbb{R}}
\def\CC{\mathbb{C}}
\def\NN{\mathbb{N}}
\def\ZZ{\mathbb{Z}}

\def\lcm{\mathrm{lcm}}
\def\uts{\left(U(t,s)\right)_{t\ge s} }
\def\xR{X_{\mathcal R}(s)}
\def\xS{X_{\mathcal S}(s)}
\def\utsR{\left(U_{\mathcal R}(t,s)\right)_{t\ge s} }
\def\utsS{\left(U_{\mathcal S}(t,s)\right)_{t\ge s} }

\title[Asymptotic periodicity of flows in time-depending networks]
      {Asymptotic periodicity of flows in time-depending networks}

\author[Fatih Bayazit, Britta Dorn and Marjeta Kramar Fijav\v{z}]{}


 \email{faba@fa.uni-tuebingen.de}
 \email{britta.dorn@uni-tuebingen.de}
 \email{marjeta.kramar@fgg.uni-lj.si}

\thanks{The last author would like to acknowledge the support of the German Academic Exchange Service (DAAD) during her stay at the University of T\"ubingen.}

\begin{document}
\maketitle

\centerline{\scshape Fatih Bayazit}
\medskip
{\footnotesize
 \centerline{Mathematisches Institut, Universit\"at T\"ubingen}
 \centerline{Auf der Morgenstelle 10, D-72076 T\"ubingen, Germany}
} 
\medskip

\centerline{\scshape Britta Dorn }
\medskip
{\footnotesize
 \centerline{ WSI f\"ur Informatik, Universit\"at T\"ubingen}
 \centerline{Sand 13, D-72076 T\"ubingen, Germany}
}

\medskip

\centerline{\scshape Marjeta Kramar Fijav\v{z}}
\medskip
{\footnotesize
 \centerline{University of Ljubljana, Faculty for Civil and Geodetic Engineering}
 \centerline{Jamova 2, SI-1000 Ljubljana, Slovenia}
 \centerline{and}
 \centerline{Institute of Mathematcs, Physics, and Mechanics}
 \centerline{Jadranska 19, SI-1000 Ljubjana, Slovenia}
} 
\bigskip

\begin{abstract}
We consider a linear transport equation on the edges of a network with time-varying coefficients. Using methods for non-autonomous abstract Cauchy problems, we obtain well-posedness of the problem and describe the asymptotic profile of the solutions under certain natural conditions on the network. We further apply our theory to a model used for air traffic flow management. 
\end{abstract}

\section{Introduction}

Dynamical processes taking place in networks have been of enormous interest  in recent years and have various applications for real life phenomena. We are interested in transport processes or {\it  flows}  in networks. Methods from the theory of operator semigroups to treat such processes were first used in \cite{KS05} for a finite network where a simple transport equation 
\[\frac{\partial}{\partial t} u(x,t)=\frac{\partial}{\partial x}  u(x,t)
\]
was considered on the edges together with boundary conditions of Kirchhoff-type in the vertices. These methods were further  applied to various generalizations of this problem in finite~\cite{Sik05,MS07,Rad08} or even infinite networks \cite{Dor08,DKS09}. The authors obtain well-posedness and describe the asymptotic behavior of the solutions. Further, \cite{EKNS08,EKKNS10} studied control problems for flows in networks. See also \cite{DKNR10} for a survey of the semigroup approach to transport processes in networks.

The processes in  all mentioned works  are autonomous, i.e. the differential operators governing the processes do not change in time. 
Motivated\footnote{We are grateful to Benedetto Piccoli for drawing our attention to this problem.} by applications to air traffic flow management (see Section~\ref{ATFM}), we now study non-autonomous processes. More precisely, we are interested in transport processes where the boundary conditions in the vertices vary in time. This yields differential operators with varying domains and the corresponding Cauchy problems become non-autonomous. Solutions to such problems are described by evolution families instead of semigroups, see~\cite{Bay12a, EN00, Nag95, NN02}. 

In the following we first define the time-depending network with a transport process in it. 
Our main tool to study such non-autonomous processes is the theory of difference evolution equations as developed in \cite{Bay12a,Bay12b} which we briefly describe in Section~\ref{diffeq}. The main results are contained in Sections \ref{flows} and \ref{ATFM} where we treat two different flow processes in a network and prove well-posedness of both of these problems. Assuming periodic boundary conditions, we obtain asymptotically periodic behavior of the solutions. The period is given in terms of the (time-depending) network structure. 


\section{Preliminaries}

\subsection{Time-depending networks}\label{networks}
The network is modeled by a finite directed graph $G$ consisting of $n$ 
vertices $\mv_1,\dots,\mv_n$ and $m$  directed edges (arcs)
$\me_1,\dots,\me_m$. We equip every edge $\me_j$ with time-varying weight $\omega_{ij}(t)\ge 0$
such that 
\begin{equation}\label{w}
\sum_{j=1}^m \omega_{ij}(t)=1 \text{ for every }t\in\RR_+ \text{{ and every }} i 
\end{equation}
(here $i$ numbers either vertices or edges --- it depends on the concrete problem and we will specify it later on).
The graph structure is described by the \emph{outgoing incidence
    matrix} $\Phi ^{-}=\left( \phi _{ij}^{-}\right) _{n\times m}$ with
    \begin{equation*}
    \phi _{ij}^{-}:= \begin{cases}
    1, & \mbox{if $\mv_i\stackrel{\me_j}{\longrightarrow}$}, \\
    0, & \text{otherwise,}
    \end{cases}
    \end{equation*}
  and the \emph{incoming incidence
    matrix}  $\Phi ^{+}=\left( \phi _{ij}^{+}\right) _{n\times m}$ with
    \begin{equation*}
    \phi _{ij}^{+}:= \begin{cases}
    1, & \mbox{if $\stackrel{\me_j}{\longrightarrow}\mv_i$}, \quad \\
    0, & \text{otherwise.}
    \end{cases}
    \end{equation*}
Instead of using incidence matrices, it is sometimes more convenient to use adjacency matrices. Here, we use the (transposed)  {\em adjacency matrix of the line graph} 
$ \bB= \left( b_{ij}\right) _{m\times m}$ with entries
\begin{equation*}\label{B}
b_{ij}:=  \begin{cases}
1, & \mbox{if $\stackrel{\me_j}{\longrightarrow}\mv\stackrel{\me_i}{\longrightarrow}$},\\
  0, & \text{otherwise.}
    \end{cases}
  \end{equation*}    

A directed graph is called {\em strongly connected} if for any pair of distinct vertices $\mv_i, \mv_j$  
there is a directed path in the graph going from $\mv_i$ to $\mv_j$ and vice versa. This property can be characterized by irreducibility of the usual vertex adjacency matrix (see e.g.~\cite[Theorem IV.3.2]{Min88}), but also by our adjacency matrix of the line graph.
\begin{lemma}{\rm \cite[Proposition 4.9]{Dor08}}\label{irreduc} A directed graph is strongly connected if and only if the  matrix $\bB$ is irreducible.
\end{lemma}

 \subsection{Transport processes}\label{transport}  In order to model a transport process on the edges, we normalize the edges as $\me_j\cong [0,1]$ and parameterize them contrary to the direction of the flow, i.e., the material flows
from~$1$ to~$0$. We consider some finite mass distributed on the edges of the network and denote by~$u_j(x,t)$ its density at position  $x \in [0,1]$ of the edge $\me_j$ and at time $t$, hence $u_j: [0,1]\times\RR\to\RR$, $j=1,\dots m$.

Our basic assumptions on the process are the following.
\begin{enumerate}
\item On each edge $\me_j$ we describe the transport process by
   \begin{equation*}
{\frac{\partial }{\partial t}u_{j}\left( x,t\right)} =
 \frac{\partial }{\partial
x}u_{j}\left( x,t\right),\quad x\in (0,1),\, t\geq s.
  \end{equation*}
\item  The initial distribution of the mass on the edges $\me_j$ at time $s\in\RR$ is given by
\begin{equation*}
u_j(x,s)=f_j(x),\quad x\in(0,1).
\end{equation*}
 \item \label{KL} No mass is gained or lost during the process. In particular,
  no absorption takes place along the edges, and in each node $\mv_i$ we have
  a \emph{Kirchhoff law}
  \begin{equation*}
  \sum_{j=1}^m \phi_{ij}^+u_j(0,t) =  \sum_{k=1}^m \phi_{ik}^-u_k(1,t),\quad  t\geq s.
\end{equation*}
\item \label{boundary} In each vertex $\mv_i$ the incoming material is
  distributed
    into the outgoing edges $\me_j$ according to the time-varying weights  $\omega_{ij}(t)\geq 0$ so that
    \eqref{w} holds.
\end{enumerate}
By choosing two different ways to assign the weights  $\omega_{ij}(t)$ to the edges,  we will in Sections \ref{flows} and \ref{ATFM} obtain two different flow processes in the network. In the first case we will assume that the material is collected in the vertex and is then redistributed according to the weights. In the second case we want to keep track of the origin of the material and hence the weights will give the proportions of the material that flows from one edge into another one.


\section{Non-autonomous difference equations}\label{diffeq}

To tackle our transport problem in time-depending networks we will use the theory of positive evolution families corresponding to a class of non-autonomous difference equations developed in \cite{Bay12a} and \cite{Bay12b}. We explain the  terminology and state the results needed below. 

Choose the Banach space $X = L^1 \left([0,1],\CC^m\right)$ as the state space of
the system. For a family of matrices $\left(B(t)\right)_{t\in\RR}\subseteq M_{m}(\CC)$ we define {\em difference operators} 
$A(t):D\left(A(t)\right)\to X$ by
\begin{equation}\label{diff-op}
D\left(A(t)\right):= \left\{f \in W^{1,1}\left([0,1],\CC^m\right)\mid f(1)= B(t) f(0)\right\}  \text{ and }A(t)f:=f' 
\end{equation}
for $f\in D\left(A(t)\right) $ and $t\in\RR$. The {\emph non-autonomous abstract Cauchy problem} corresponding to the operators $\left(A(t) ,D\left(A(t)\right)\right)$ is of the form
\begin{equation*} \label{acp}
(nACP)\left\{ \begin{aligned}\dot{u}\left( t\right) &=A(t)u(t), \quad t\ge s, \\ 
u(s) &= f_s \in X. \end{aligned}\right.
\end{equation*}
A {\em classical solution} to the $(nACP)$ is a differentiable function $u\in C^1\left([s,\infty), X\right)$ such that $u(t)\in D\left(A(t)\right)$ for every $t\ge s$  and $u$ satisfies $(nACP)$. Furthermore, we say that $(nACP)$ is {\em well-posed} if there exists a unique evolution family $\left(U (t, s)\right)_{t\ge s}$ such that that the regularity subspaces 
\begin{equation*}
Y_s := \left\{ f\in X \mid [s,\infty) \ni t \mapsto U(t, s) f \text{ is a classical solution to }(nACP) \right\}
\end{equation*}
are dense in $X$ for every $s\in\RR$. For the definition of the evolution family see \cite[Section 2]{Bay12a} or \cite[Definition VI.9.2]{EN00}. We also recommend \cite{Nag95}, \cite{NN02}, or \cite[Chapter 5]{Paz83} for further information on evolution families and their relation to non-autonomous Cauchy problems.

Since the domains $D\left(A(t)\right) $ are time-dependent and do not contain a common 
core, none of the usual well-posedness results is applicable in our case. We will use the following results from \cite{Bay12a} instead. 

\begin{proposition}\label{wp} {\rm \cite[Theorem 2]{Bay12a}} Let the mapping $t\mapsto B(t)$ be uniformly bounded and absolutely continuous.
Then the (nACP) associated to the operators $\left(A(t),D\left(A(t)\right)\right)$ given by \eqref{diff-op} is well-posed. 
\end{proposition}
In \cite{Bay12a} even an explicit formula for the corresponding evolution family is given. We state here this formula in a special case.

\begin{proposition}  {\rm  \cite[Equation~(8)]{Bay12a} } \label{formula} Let the mapping $t\mapsto B(t)$ be uniformly bounded, absolutely continuous and 1-periodic, i.e. $B(t+1)=B(t)$ for every $t\in\RR$. Then the unique classical solution to $(nACP)$ is given by $u(t)=U(t,s)f_s$, where 
\begin{equation}\label{evol-fam-formula}
(U(t,s)f)(x)=B^k(t+x)f(x+t-s-k),
\end{equation}
for $ f\in X$, $ x\in[0,1],$ $k\le x+t-s < k+1$ and $k\in\NN_0$.
\end{proposition}

In order to study the asymptotic behavior of the solutions, some more regularity assumptions are needed.
Denote the unit circle by $\Gamma:=\{z\in\CC\mid |z|=1\}$. 

\begin{proposition}\label{asy}
Let  $t\mapsto B(t)$ be an absolutely continuous 1-periodic mapping and let $B(t)$ be a stochastic irreducible matrix for every $t\in\RR$.  Then there is a family of projections $\left\{P(s)\mid s\in\RR\right\}$ in $ {\mathcal L}(X)$, commuting with the evolution family $\uts$ and decomposing the space $X$  as
\begin{equation*}
X = \xR\oplus \xS:=P(s)X\oplus \ker P(s),
\end{equation*}
such that the following properties hold.
\begin{enumerate}
\item[(i)] The subspaces $\xR$ and $\xS$ are $\uts$-invariant for every $s\in\RR$.
\item[(ii)] $\utsS:=\left(U(t,s)|_{\xS}\right)_{t\ge s}$ is uniformly exponentially stable, i.e. there exist $C\ge 1$ and $\omega>0$ such that
\begin{equation*}
\| U_{\mathcal S}(t,s)\| \le C e^{-\omega (t-s)},\quad t\ge s.
\end{equation*}
\item[(iii)] $\utsR := \left(U(t,s)|_{\xR}\right)_{t\ge s}$ can be extended to an invertible evolution family $\left(U_{\mathcal R}(t,s)\right)_{(t,s)\in\RR^2}$ which
is positive and periodic in evolution, i.e. $U_{\mathcal R}(s+\tau,s) = I_{\xR}$ for every $s\in\RR$, with the period
\begin{equation*}\label{period}
\tau = \lcm\left\{\left|  \sigma\left(B(t)\right)\cap\Gamma \right| \mid t\in[0,1]\right\}, 
\end{equation*}
where $\lcm$ denotes the least common multiple, and $| A |$ stands for the number of points of the set $A$.
\item[(iv)] For $\tau$ as above there exists a $\tau$-periodic positive group $\left(T(t)\right)_{t\in\RR}$ such that
\begin{equation*}
\| U (t, s) -T(t-s)P (s)\|\stackrel{t\to\infty}{\longrightarrow}{0}
\end{equation*}
for every $s\in\RR$.
\end{enumerate}
\end{proposition}
\begin{proof}
The $m\times m$ matrices $B(t)$ are all stochastic and irreducible, therefore by Perron-Frobenius theory (see \cite[Theorem I.6.5]{Sch74}) the peripheral spectrum $\sigma\left(B(t)\right)\cap\Gamma$, for every $t\in\RR$, is a finite group consisting of  (at most $m$) roots of unity which are all first order poles of the resolvent. Hence the union
$$\bigcup_{t\in [0,1]} \left\{ \sigma\left(B(t)\right)\cap\Gamma\right\}$$ is a finite discrete set and
the least common multiple~$\lcm$ in (iii) is well defined.
The stochasticity of the matrices $B(t)$ also implies that the evolution family $\uts$  given in \eqref{evol-fam-formula} consists of positive contractions. Combining  Proposition 6, Definitions 7 and 8, and Theorem 9 from  \cite{Bay12a} we now obtain the decomposition of the space $X$ with properties (i) and (ii). Finally,  (iii) follows from \cite[Theorem 25]{Bay12a} and (iv) from \cite[Theorem 26]{Bay12a}.
\end{proof}


\section{Flows in nonautonomous networks}\label{flows}

Consider now a finite weighted network $G$ as in  Section \ref{networks} with incidence matrices $\Phi^-$ and $\Phi^+$. The time-dependent weights $\omega_{ij}(t)\ge 0$ in every vertex $\mv_i$ give the proportions of the incoming material to be distributed into the outgoing edges $\me_j$ at time $t$, where
\begin{equation*}\omega_{ij}(t)\equiv 0 \text{ if } \phi^-_{ij}= 0.
\end{equation*}
This condition reflects the fact that the edges of our network are fixed and the flow takes place only on the edges of the network. Note however that it might happen that no material is sent from the vertex $\mv_i$ into the edge $\me_j$ at time $t_0$ for some $t_0$,  meaning that $\omega_{ij}(t_0)=0$ even if $\phi^-_{ij}\ne 0$.
We store the weights in the  time-dependent \emph{weighted outgoing incidence matrix}   $\Phi_w^{-}(t)=\left( \phi_{w,ij}(t)\right) _{n\times m}$ defined as
\begin{equation*}
    \phi_{w,ij}^{-}(t):= \begin{cases}
   \omega_{ij}(t), & \mbox{if $\mv_i\stackrel{\me_j}{\longrightarrow}$}, \\
    0, & \text{otherwise.}
    \end{cases}
    \end{equation*}
The $m\times m$  time-dependent \emph{weighted adjacency matrix of the line graph} $ \bB_w(t)$ is  obtained by
\begin{equation}\label{bbt}
\bB_w(t):= \left(\Phi_w^{-}(t)\right)^T \Phi ^{+}.    
\end{equation}
Note that the nonzero entries of $\bB_w(t)$ are in one-to-one correspondence with the nonzero entries of the unweighted adjacency matrix~$\bB$ of the line graph. 
We assume that there is no absorption in the vertices, hence \eqref{w} holds for every $i\in\{1,\dots,n\}$ and all $t\in\RR_+$, and the matrices  $\bB_w(t)$ are all column-stochastic. 

By $G_t$ we will denote the network  at time $t$ obtained from the adjacency matrix $\bB_w(t)$. This means that $G_t\subseteq G$ where the edges of $G$ with no inflow at time $t$ are deleted.

Under these assumptions we study the following transport process in~$G$.
\begin{equation*}
(nF)\left\{
\begin{tabular}{rcll}
${\frac{\partial }{\partial t}u_{j}\left( x,t\right)} $ & $=$ & $
 \frac{\partial }{\partial
x}u_{j}\left( x,t\right) ,\, x\in (0,1),\, t\geq s,$ &\\[0.5em]
$ u_{j}\left( s,0\right) $ & $=$ & $ f_{j}\left(s\right) ,\, s\in (0,1),$ & $(IC)$ \\[0.5em]
$  \phi_{ij}^{-}u_{j}\left( 1,t\right) $ & $=$ & $
\omega_{ij}(t)\sum_{k=1}^m\phi _{ik}^{+}u_{k}\left(0,t\right) ,\, t\geq 0$
& $(nBC)$
\end{tabular}
\right.
\end{equation*}
for $i= 1, \dots, n$ and $j= 1, \dots, m$. It is of the form given in Section \ref{transport}. Note that the non-autonomous boundary conditions $(nBC)$ together with \eqref{w} imply the Kirchhoff law~(\ref{KL}).

In order to use the results from Section \ref{diffeq}, we now take the Banach space $X = L^1\left([0,1],\CC^m\right)$ and the difference operators $A_{\bB}(t)$ on $X$ associated to the family of matrices $\left(\bB_w(t)\right)_{t\in\RR}$ as defined in \eqref{diff-op}, hence
\begin{equation*}
A_{\bB}(t)f:=f' \text{ with domain } D\left(A_{\bB}(t)\right):= \left\{f \in W^{1,1}\left([0,1],\CC^m\right)\mid f(1)= \bB_w(t) f(0)\right\}. 
\end{equation*}

\begin{proposition}\label{bc-da}
The non-autonomous abstract Cauchy problem $(nACP)$ corresponding to $\left(A_{\bB},D\left(A_{\bB}(t)\right)\right)$ is an abstract version of the transport process in the time-depending network $(nF)$.
\end{proposition}
\begin{proof}
We only need to observe that the non-autonomous boundary conditions $(nBC)$ of the problem $(nF)$ are hidden in the domain $D\left(A_{\bB}(t)\right)$:
\begin{equation*}
g \in W^{1,1}\left([0,1],\CC^m\right) \text{ satisfies } (nBC) \iff g(1)=\bB_w(t) g(0),
\end{equation*}
similarly as in~\cite[Prop.\ 3.1]{Dor08}.
\end{proof}

The well-posedness of the analogous problem in the autonomous case (even for infinite networks) accompanied with an explicit formula for the solution was shown in \cite[Prop.\ 3.3]{Dor08}.
By Propositions~\ref{wp} and~\ref{formula} we can prove a well-posedness result of our non-autonomous flow problem $(nF)$ for a special class of periodic time-dependent networks. 

\begin{corollary}\label{cor:well-posed}
If the mappings $t\mapsto\omega_{ij}(t)$ are absolutely continuous, then $(nF)$ is well-posed. \\
In particular, if in addition, the mappings $t\mapsto \bB_w(t)$ are 1-periodic, i.e. $\bB_w(t+1)=\bB_w(t)$ for all $t\in\RR$, then the unique classical solution $t\mapsto u(t,\cdot)$ to the non-autonomous flow problem $(nF)$ in $G$ is obtained by the evolution family in~(\ref{evol-fam-formula}) as 
\begin{equation}\label{evol-fam}
u(t,x)=\left(U(t,s)f_s\right)(x) = \bB_w^k(t+x)f_s(x+t-s-k),
\end{equation}
where $f_s$  is given by the initial conditions $(IC)$, $t\geq s$, $ x\in[0,1]$,  $k\le x+t-s < k+1$, and $k\in\NN_0$.
\end{corollary}
     
Note that the period is assumed to be 1 only as a matter of convenience. We could take any natural number and modify the above formula appropriately. \\
Using Proposition \ref{asy} we obtain an asymptotically periodic behavior of our non-autonomous flows.  

\begin{theorem}\label{asy-flow}
Let the network $G_t$ be strongly connected for every $t\in\RR$ and let the mapping $t\mapsto \bB_w(t)$ be absolutely continuous and 1-periodic.
Then the flow evolution family \eqref{evol-fam} converges uniformly to a periodic  positive group in the sense of Proposition \ref{asy}.$(iv)$. Its period $\tau$ can be computed as  
\begin{equation*}
\tau=\lcm\left\{\gcd\{l\mid \me_{j_1},\dots \me_{j_l} \text{ form a cycle in }G_t\}\mid t\in \RR\right\}, 
\end{equation*}
where $\gcd$ denotes the greatest common divisor.
\end{theorem}

\begin{proof} By \eqref{w}, the matrices  $\bB_w(t)$ are all column-stochastic. Since the graphs $G_t$ are all strongly connected, the matrices~$\bB_w(t)$ are all irreducible by Lemma~\ref{irreduc}.  Hence we can apply Proposition \ref{asy}.
For the expression for the period $\tau$ note that $|\sigma\left(\bB_w(t)\right)\cap\Gamma|$ equals the index of imprimitivity of the matrix $\bB_w(t)$ (see \cite[Definition III.1.1]{Min88}), which is the same as the greatest common divisor of all cycle lengths in the network $G_t$, cf.~\cite[Theorem IV.3.3]{Min88}.  \end{proof}

\begin{remark}\label{rem:weights-positive}
If we assume that the nonzero weights remain strictly positive in time, i.e.
\begin{equation*}
\phi^-_{ij}\ne 0 \Longrightarrow\omega_{ij}(t)\ne 0\quad\text{ for all }t\in\RR,
\end{equation*}
then the index of imprimitivity of $\bB_w(t)$ does not change in time either, and the period $\tau$ can be computed simply as
\begin{equation*}
\tau=\gcd\{l\mid \me_{j_1},\dots \me_{j_l} \text{ form a cycle in }G\}.
\end{equation*}
This means that in the case of non-disappearing edges in the network, the strictly positive weights do not have any impact on the period and the asymptotic behavior remains the same as in the autonomous case (see \cite[Corollary 4.7]{KS05}).
\end{remark} 

      
\subsection{Examples}

\begin{enumerate}
\item We consider the family of networks $G_t$ depicted in Figure~\ref{example-1}. The weights on the edges~$\me_1, \, \me_2,\,  \me_3, \, \me_6$ are constant, the weights on the edges $\me_4$ and $\me_5$ vary in time. Each edge has a nonzero flow of material on it at every time. The corresponding adjacency matrix~$\bB_w(t)$ is 

$$\bB_w(t) =\begin{pmatrix} 0 & 0 & 0 & 1 & 0 & 0 \\
 1 & 0 & 0 &0 & 0 & 0  \\
 0 & 1 & 0 &0 & 0 & 1  \\
 0 & 0 & \frac 14 +\frac 12 \cos^2(\pi t) &0 & 0 & 0 \\[1mm]
 0 & 0 & \frac 14 +\frac 12 \sin^2(\pi t) &0 & 0 & 0 \\
 0 & 0 & 0 &0 & 1 & 0 
           \end{pmatrix}.$$ 
           
The mapping $t \mapsto \bB_w(t)$ is 
absolutely continuous, so by Corollary~\ref{cor:well-posed}, the corresponding non-autonomous network flow problem $(nF)$ is well-posed. 
The matrices $\bB_w(t)$ are all column stochastic, 1-periodic, and, since the graphs $G_t$ are strongly connected for all~$t \in \RR$, irreducible. We can therefore apply Theorem~\ref{asy-flow}. Note that we are in the case of Remark~\ref{rem:weights-positive} since each edge carries a nonzero flow of material at every time~$t \in \RR$. For every~$t$, the graph $G_t$ contains a cycle of length~$3$ and a cycle of length~$4$. 
Hence, the flow evolution family converges uniformly to a periodic positive group with period $\tau = \gcd\{3,4\}=1$ in the sense of Proposition~\ref{asy}.(iv), meaning that the flow on this graph asymptotically behaves periodically with period~$\tau=1$.

\begin{figure}
\begin{center}
\includegraphics[scale=0.4]{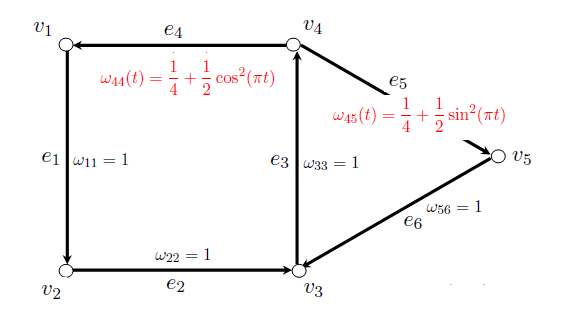}
\end{center}
\caption{The flow on this network is asymptotically periodic with period~$1$.} \label{example-1}
\end{figure}

\medskip
\item 
The networks in Figure~\ref{example-2} all belong to the family of networks~$G_t$ given by the adjacency matrix  $\bB_w(t)$ equaling
\begin{center}
\includegraphics[scale=0.45]{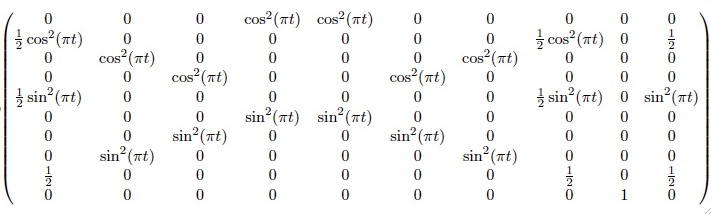}
\end{center}       
Note that the weights on edges~$\me_9$ and~$\me_{10}$ do not vary in time, but those on the remaining edges do. This is indicated in Figure~\ref{example-2} (i). Hence, some edges do not carry any flow for some $t\in \RR$: For $t \in \ZZ$, the weights on the edges~$\me_5, \,  \me_6,\,  \me_7, \, \me_8$ are zero (Figure~\ref{example-2} (ii)), whereas for $t \in \frac 12\ZZ\setminus\{0\}$, there is no flow on the edges~$\me_1, \, \me_2,\,  \me_3, \, \me_4$ (Figure~\ref{example-2} (iii)). Since~$\me_9, \, \me_{10}$ form a cylce of length~$2$, the greatest common divisor of all cycle lengths in the network is equal to 2 for all times~$t \in \RR$, and hence also the least common multiple appearing in Theorem~\ref{asy-flow} is equal to 2. 
 
\begin{figure}
\begin{center}
\includegraphics[scale=0.35]{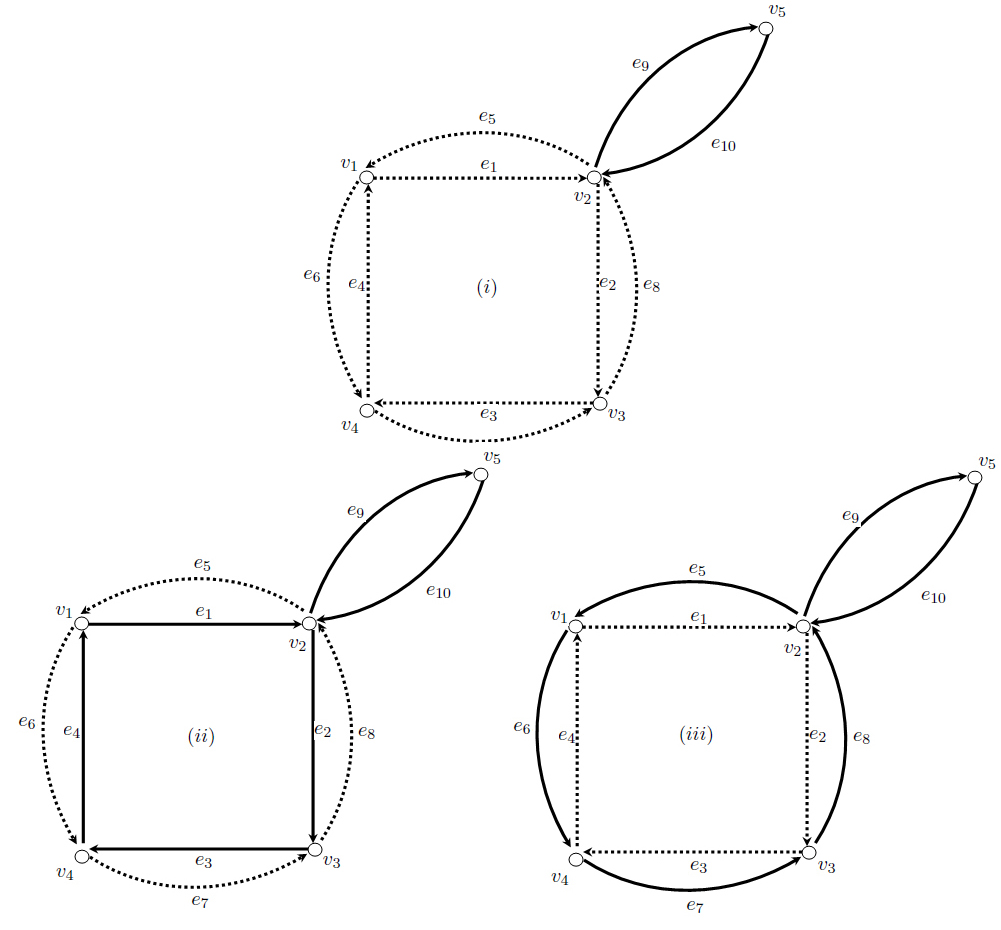}
\end{center}
\caption{The three possible states of the family of networks~$G_t$ in Example~2: Either, all edges of the network carry flow (Figure (i)), or there is no flow on the edges belonging to the outer cycle of length~$4$ (Figure (ii)) or to the inner cycle of length~$4$ (Figure (iii)).} \label{example-2}
\end{figure}

Again, we obtain well-posedness of the corresponding flow problem $(nF)$ by Corollary~\ref{cor:well-posed} since the mapping $t \mapsto \bB_w(t)$ is continuously differentiable. Note that the graphs~$G_t$ are strongly connected for all~$t \in \RR$ and fulfill all assumptions required by Theorem~\ref{asy-flow}, implying that the flow evolution family converges uniformly to a periodic positive group with period $\tau=2$.

\end{enumerate}


\section{Application to air traffic flow management}\label{ATFM}

The following application of our results is motivated by the (real world) regulation of air traffic, called  {\em Air Traffic Flow Management}. Its goal is to optimize air traffic flow, i.e.,
limiting the density of aircraft in certain regions of airspace as well as operating efficient
routes subject to weather constraints. These tasks are currently prescribed by playbooks
established over time and based on controller experience. However, one of the aims 
 consists in providing a mathematical model of air traffic flow allowing to  apply mathematical
 control techniques.

\subsection{Modelling air traffic flow}
Different mathematical models for optimization strategies have been elaborated. One approach is a {\em Eulerian} model advocated
by Menon et al., see~\cite{MSB04}, where the airspace is divided into line elements corresponding to portions of airways
on which the density of aircraft can be described as a function of time and of the coordinate along the line. This approach
 focuses on the conservation of aircraft on the line elements and uses partial differential equations to describe the time
 evolution of the process. The equations used in this model also appear naturally in highway traffic and were introduced by
 Lighthill-Whitham \cite{LW56} and Richards  \cite{Ric56}.
This Eulerian network model of air traffic flow has been considered
in several works, e.g. \cite{MSB04}, \cite{BRT04a, BRT04b, BRT06}, \cite{SM05}, 
\cite{RSWB06}, \cite{SSB07}.
We also refer to the monograph by M.~Garavello and B.~Piccoli \cite{GP06} where networks
of interconnected roads are modeled and studied, and where the Lighthill-Whitham-Richards model is considered on network structures (junctions). 

We use a simplified linear Eulerian network approach. This fits into our scenario
since the traffic flow is considered as a transport process of
aircraft along the edges of a directed graph with boundary
conditions in the vertices. In this context, the vertices correspond
to different destinations (or airports) or to bifurcation points of
routes in the sky, and the edges model the given connections between
them (the above mentioned line elements).

\subsection{The allocation matrix}\label{alloc-mat}
In the  literature (e.g., \cite{BRT06}, \cite{SSB07}), the
transport processes are usually studied only on an isolated junction
of the network. An example is given in Figure~\ref{junction},
showing a junction with two incoming edges $\me_1, \me_2$ (called
{\em links} in \cite{SSB07}) and three outgoing edges $\me_3, \me_4,
\me_5$.

\begin{figure}[H]
\begin{center}
\includegraphics[width=0.4\textwidth]{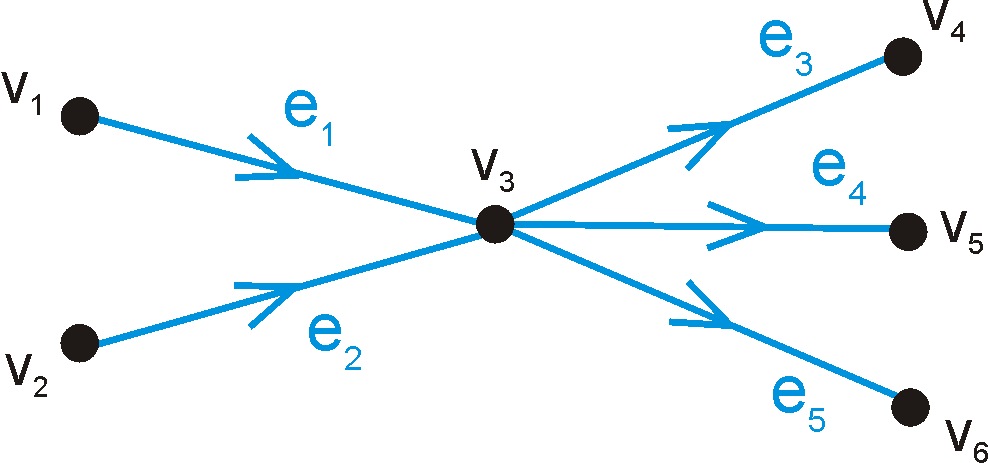}
\end{center}
\caption{A junction with two incoming and three outgoing edges.}
\label{junction}
\end{figure}

The relation between the incoming and outgoing air traffic flow at a
junction is prescribed by the so-called {\em junction allocation matrix }
$M(t) = \left(m_{ij}(t)\right)$ for $1\leq i \leq p$ and $p+1 \leq j
\leq p+q$, where $0 \leq m_{ij}(t) \leq 1 $ denotes the proportion
of aircraft from incoming link $i$ going to the outgoing link $j$ at
time $t$, and 
\begin{equation}\label{allocation-sum}
\sum_{j= p+1}^{p+q} m_{ij}(t) = 1
\end{equation} is required for
all $1 \leq i \leq p$ and $t\ge 0$ (see e.g. \cite[Section 2.3]{SSB07} for this
definition). 

As an example, for the allocation matrix $M(t_0)$ of the junction in Figure~\ref{junction} at time~$t_0$, we choose
 \begin{eqnarray*} & & \begin{array}{ccccc} \me_3 \quad & \me_4 \quad & \me_5 \quad \end{array}  \\
                                     & & \begin{array}{ccccc} \quad \uparrow  & \quad \uparrow & \quad \uparrow \end{array} \\
                                 M(t_0) &=&  \begin{pmatrix}
                                                                                        1/2 & 1/3& 1/6\\
                                                                                0 & 1/4  & 3/4
                                    \end{pmatrix}   \begin{array}{cc} \leftarrow & \me_1 \\
                                                                                                        \leftarrow & \me_2\end{array}
 \end{eqnarray*}
This means that at time~$t_0$, half of the airplanes arriving from edge $\me_1$
continue their way on edge $\me_3$, one third of them chooses
$\me_4$ and the remaining ones travel to edge $\me_5$, whereas none
of the airplanes coming from edge $\me_2$ go in the direction of
$\me_3$, but one forth of them to $\me_4$ and the remaining
three-fourths to $\me_5$. Note that \eqref{allocation-sum} holds. 

\subsection{Our setting and results}
We now consider a strongly connected directed graph $G$ consisting of $n$ vertices and $m$ edges, and some air traffic flow on it, which, according to the linear Eulerian model, can be considered as a transport process. The boundary conditions of this process are given in the (transposed) {\em network allocation matrix} which we now define on the whole network (not only on a single junction) as 
$$\bM(t) := \left(m_{kl}(t)\right) \text{ for } k,\, l \in \{1, \dots, m\}, \, t  \in \RR, $$
where $0 \leq m_{kl}(t)\leq 1$ denotes the proportion of aircrafts arriving from edge $\me_l$ leaving into edge~$\me_k$ at time~$t$. We imply that the flow only takes place on the edges of the network $G$ which is given by the (transposed) adjacency matrix of the line graph $\bB=(b_{kl})_{m\times m}$ and we set
\begin{equation*}
m_{kl}(t) \equiv 0 \text{ if } b_{kl}=0.
\end{equation*}
We further require that
\begin{equation}\label{m-sum}
\sum_{k= 1}^{m} m_{kl}(t) = 1\quad\text{for all } l\in\{1,\dots,m\} \text{ and  } t\ge 0.
\end{equation}
This assumption corresponds to Equation~\eqref{w} (the index~$l$ now runs over the edges of the graph) and  makes the allocation matrix $\bM(t)$ column stochastic. 

Since every edge of the network only has one end point and one starting point, the transpose $M^T(t)$ of every junction allocation matrix $M(t)$ given in Section~\ref{alloc-mat} is a submatrix of the bigger  matrix $\bM(t)$. Hence  our network allocation matrix $\bM(t)$ contains all the information stored in the separate junction allocation matrices and 
\eqref{m-sum} corresponds to  the equations  (\ref{allocation-sum}) in every junction.

We now model the transport process in the network as in Section \ref{transport} and obtain 
 the following air traffic flow problem. 
\begin{equation*}
(ATF)\left\{
\begin{tabular}{rcll}
${\frac{\partial }{\partial t}u_{j}\left( x,t\right)} $ & $=$ & $
 \frac{\partial }{\partial
x}u_{j}\left( x,t\right) ,\, x\in (0,1),\, t\geq s,$ &\\[0.5em]
$ u_{j}\left( s,0\right) $ & $=$ & $ f_{j}\left(s\right) ,\, s\in (0,1),$ &  \\[0.5em]
$  u_{j}\left( 1,t\right) $ & $=$ & $
\sum_{k=1}^m m_{jk} (t) u_{k}\left(0,t\right) ,\, t\geq 0,$
& 
\end{tabular}
\right.
\end{equation*}
for  $j= 1, \dots, m$. Observe that our non-autonomous boundary conditions together with~\eqref{m-sum} imply the Kirchhoff law~(\ref{KL}) in the vertices.

We now proceed as in Section \ref{flows}. Problems $(nF)$ and $(ATF)$ have different solutions since their boundary conditions differ. The conditions in $(ATF)$ contain more information and are more demanding. We could look at the problem $(ATF)$ as a subproblem of $(nF)$ which can also be implemented by constructing a larger graph (the precise implementation is done in \cite[Section 7.3.2]{dissBritta}).
In this way we would easily obtain well-posedness, however the formulae for the solutions and period would relate to the artificially created larger network and it would be difficult to relate it to the original problem. Therefore we rather repeat the steps taken for $(nF)$ in Section \ref{flows}, now for the problem $(ATF)$ instead. We will see that the main difference is that the weighted adjacency matrix $\bB_w(t)$ is replaced by the network allocation matrix $\bM(t)$. These matrices are different but share many important properties (such as positivity, irreducibility, etc.).

We will assume that the entries of $\bM(t)$ vary in an absolutely continuous and periodic way. This assumption is natural if we think of periodically changing flight schedules (day-night rhythms, daily or weekly periods), and without loss of generality we may assume that the period is 1. We also assume that the network $G_t$ remains strongly connected at all times $t$, even if some edges of~$G$ might not carry any flow at certain times.

The state space of this system can be modeled by the Banach space 
$X = L^1 \left([0,1],\CC^m\right),$ and the  transport process can then be described via the difference operators
$$A_{\bM}(t):D\left(A_{\bM}(t)\right)\to X$$ defined by
\begin{equation*}\label{allocation-diff-op}
A_{\bM}(t)f:=f' \text{ with domain } D\left(A_{\bM}(t)\right):= \left\{f \in W^{1,1}\left([0,1],\CC^m\right)\mid f(1)= \bM(t) f(0)\right\}  
\end{equation*}
for $t\in\RR$. 
As in Proposition~\ref{bc-da}, we can see that the non-autonomous abstract Cauchy problem 
\begin{equation*} \label{allocation-acp}
\left\{ \begin{aligned}\dot{u}\left( t\right) &=A_{\bM}(t)u(t), \quad t\ge s, \\ 
u(s) &= f_s \in X. \end{aligned}\right.
\end{equation*}
corresponds to the transport problem (ATF).
Applying our results from Section~\ref{diffeq}, we obtain the following well-posedness result together with a description of the asymptotic shape of the solutions. 

\begin{theorem}\label{thm:ATFM}
Let  $t\mapsto \bM(t)$ be an absolutely continuous  $1$-periodic mapping. For every $t\ge 0$ let the graphs $G_t$ be strongly connected and the matrices $\bM(t)$ column stochastic. Then the non-autonomous transport problem $(ATF)$ is well-posed. Its unique classical solution $t\mapsto u(t,\cdot)$ is given by the flow evolution family as 
\begin{equation}\label{evol-fam-ATFM}
u(t,x)=\left(U(t,s)f_s\right)(x) = \bM^k(t+x)f_s(x+t-s-k),
\end{equation}
where $f_s$  is the initial distribution of aircraft flow, $t\geq s$, $ x\in[0,1]$,  $k\le x+t-s < k+1$, and $k\in\NN_0$.

The flow evolution family converges uniformly to a periodic positive group in the sense of Proposition \ref{asy}.(iv)  with period 
\begin{equation*}
\tau=\lcm\left\{\gcd\{l\mid \me_{j_1},\dots \me_{j_l} \text{ form a cycle in }G_t\}\mid t\in \RR\right\}.
\end{equation*}

\end{theorem}

\begin{example}
A small example is shown in Figure~\ref{example-3}. In this network, one third of the flow arriving from edge~$\me_1$ is continuing its way into edge~$\me_3$, the remaining proportion of two thirds flows into the edge~$\me_4$. The proportion of the flow arriving from edge~$\me_2$ and continuing into edges~$\me_3$ and~$\me_4$, respectively, varies in time. The corresponding network allocation matrix is
$$\bM(t) =\begin{pmatrix} 0 & 0 & 0 & 0 & 1 & 0 \\
 0 & 0 & 0 &0 & 0 & 1  \\
\frac 13 &\frac 14 +\frac 12 \cos^2(\pi t) & 0 &0 & 0 & 0  \\[1mm]
\frac 23 & \frac 14 + \frac 12 \sin^2(\pi t)  &0 & 0 & 0 & 0\\
 0 & 0 & 1 &0 & 0 & 0 \\
 0 & 0 & 0 &1 & 0 & 0 
\end{pmatrix}.$$
Note that the allocation matrix contains more information on the flow than the corresponding adjacency matrix $\bB_w(t)$ used in Section \ref{flows}.

According to Theorem~\ref{thm:ATFM}, the air traffic flow problem $(ATF)$ on this network is well-posed. The flow evolution family describing the solutions to the problem is given by the powers of $\bM(t)$ and the initial distribution of aircraft flow as in Equation~\eqref{evol-fam-ATFM}, and converges uniformly to a periodic positive group with period~$\tau=3$.  

\begin{figure}
\begin{center}
\includegraphics[scale=0.25]{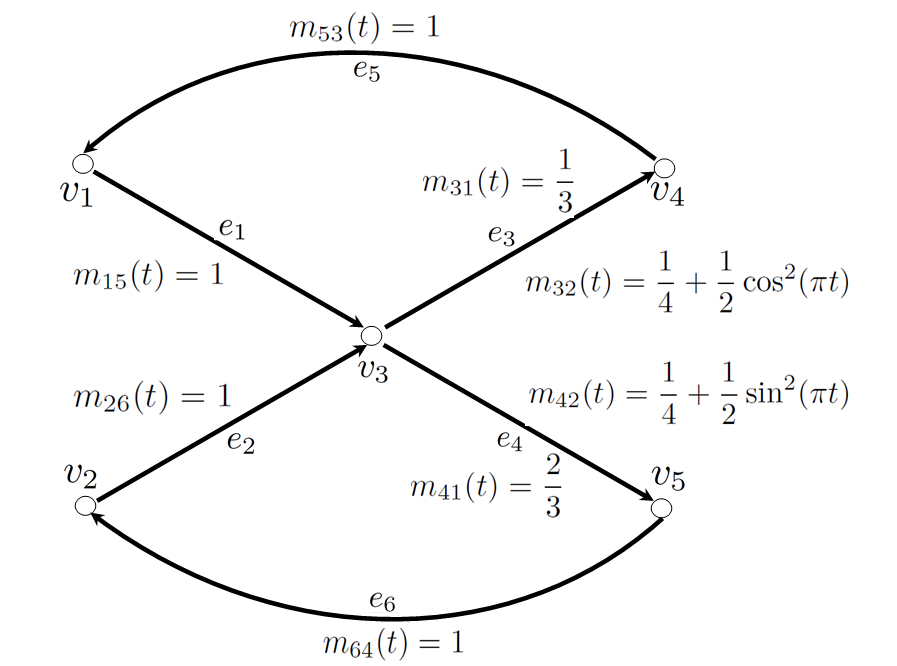}
\end{center}
\caption{A toy example to illustrate Theorem~\ref{thm:ATFM}.} \label{example-3}
\end{figure}
\end{example}

\begin{remark}
All our results are obtained under the assumption on absolute continuity of the time-varying traffic distribution coefficients.
We are aware of the limitations  this condition poses for the real-life applications. 
 We believe one can reformulate our results for the case of  piecewise absolutely continuous or even only measurable coefficients. This would however demand an appropriate formulation and corresponding proofs of the abstract results on evolution families we use from \cite{Bay12a}.  Therefore we leave this task for our future considerations.  
\end{remark}

\medskip

\medskip

\end{document}